\documentclass[psamsfonts]{amsart}

\usepackage{amssymb,amsfonts}
\usepackage[all,arc]{xy}
\usepackage{enumerate}
\usepackage{mathrsfs}
\usepackage{graphicx}

\newtheorem{theorem}{Theorem}[section]
\newtheorem{corollary}[theorem]{Corollary}
\newtheorem{proposition}[theorem]{Proposition}
\newtheorem{lemma}[theorem]{Lemma}

\theoremstyle{definition}
\newtheorem{definition}[theorem]{Definition}

\newtheorem{example}[theorem]{Example}
\newtheorem{examples}[theorem]{Examples}

\newtheorem{remark}[theorem]{Remark}

\DeclareMathOperator{\Nil}{Nil}

\DeclareMathOperator{\End}{End}
\DeclareMathOperator{\Id}{Id}
\DeclareMathOperator{\Max}{Max}
\DeclareMathOperator{\Spec}{Spec}

\DeclareMathOperator{\Ann}{Ann}


\numberwithin{equation}{section}

\begin{document}

\title[Semirings in which every non-unit is a zero-divisor]{Classical semirings}

\author[H. Behzadipour]{Hussein Behzadipour}

\address{Hussein Behzadipour\\
	Department of Mathematics \\
	Sharif University of Technology\\
	Tehran\\
	Iran}
\email{hussein.behzadipour@gmail.com}

\author[H. Koppelaar]{Henk Koppelaar}

\address{Henk Koppelaar\\
	Faculty of Electrical Engineering, Mathematics and Computer Science\\
	Delft University of Technology\\ Delft\\ The Netherlands}

\email{koppelaar.henk@gmail.com}

\author[P. Nasehpour]{Peyman Nasehpour}

\address{Peyman Nasehpour\\
	An independent researcher of mathematics and music\\ Tehran, Iran}

\email{nasehpour@gmail.com}



\begin{abstract}
In this paper, we investigate semirings whose elements are either units or zero-divisors (nilpotents) with many examples. While comparing these semirings with their counterparts in ring theory, we observe that their behavior is different in many cases. 
\end{abstract}

\maketitle

\section{Introduction}

Monoids with this property that each element is either a unit or a zero-divisor have applications in automata theory (see Theorem 7.26 and Example 7.5 in \cite{SokratovaKaljulaid2000}). Rings with the same property, often called classical rings, have applications in coding theory, cryptography, and system theory (see Remark \ref{applicationsclassicalrings}). A specular kind of classical rings, i.e., the algebra of dual numbers (check Definition \ref{generalizeddualnumbers} and Theorem \ref{generalizeddualnumberalgebraclassical}), is a powerful mathematical tool for kinematic and dynamic analysis of spatial mechanisms (cf. \cite{PennestriStefanelli2007}). Also, note that a subfamily of classical rings, so-called completely primary rings, have found some interesting applications in algebraic coding theory (cf. \cite{DoughertyYildizKaradeniz2011}). All these examples motivated the authors of the current paper to investigate this property in the context of semiring theory. Since the term ``semiring'' has different meanings in the literature, it is important to clarify what we mean by this term from the beginning. 

In the current paper, an algebraic structure $(S,+,\cdot,0,1)$ is a semiring if $(S,+,0)$ and $(S,\cdot,1)$ are commutative monoids, the multiplication distributes on the addition, and the element zero is absorbing for the multiplication, i.e., $s \cdot 0 = 0$, for all $s \in S$. Note that semirings are a generalization of rings and bounded distributive lattices. Those semirings that are not rings are called proper. 

For this paper, it is also crucial to recall the following concepts for a semiring $S$:

\begin{itemize}
	\item An element $s$ of $S$ is multiplicatively cancellative (regular) if \[sx_1 = sx_2~(s x_1 = 0) \implies x_1 = x_2~(x_1 = 0), \qquad\forall~x_1, x_2 \in S.\]
	
	\item An element $s$ in $S$ is a zero-divisor (nilpotent) if it is not multiplicatively regular ($s^n = 0$, for some positive integer $n$). 
\end{itemize}

Let $S$ be a semiring and consider the following conditions:

\begin{enumerate}
	\item $S$ is completely primary (i.e., each element of $S$ is either a unit or nilpotent);
	\item $S$ is classical (i.e., each element of $S$ is either a unit or a zero-divisor);
    \item Each element of $S$ is either a unit or multiplicatively non-cancellative. 
\end{enumerate}

Then, we have $(1) \implies (2) \implies (3)$ and non of the implications is reversible in general (see Proposition \ref{latticesnotclassical} and Theorem \ref{directproductcompletelyprimarysemirings}).

The main purpose of the current paper is to investigate the family of ``classical semirings'' and its subfamily ``completely primary semirings''. Here is a brief sketch of the contents of our paper:

In \S\ref{sec:nonunitsarenoncancellative}, we investigate those semirings in which every non-unit is multiplicatively non-cancellative and in Proposition \ref{piregularpro}, we prove that $\pi$-regular semirings have this property. Inspired by ring theory, we define a semiring $S$ to be $\pi$-regular if for any $s \in S$, there is a $t \in S$ and a positive integer $n$ such that $s^{n+1} t = s^n$. Note that the family of $\pi$-regular semirings include periodic semirings (see Definition \ref{piregulardef} and Examples \ref{periodicsemirings}).

Let $S$ a semiring. We recall that a commutative monoid $(M,+,0_M)$ is said to be an $S$-semimodule if there is a function, called scalar multiplication, $\lambda: S \times M \rightarrow M$, defined by $\lambda (s,m)= s\cdot m$ such that the following conditions are satisfied:

\begin{itemize}
	\item $s\cdot (m+n) = s\cdot m+s\cdot n$ for all $s\in S$ and $m,n \in M$;
	\item $(s+t)\cdot m = s\cdot m+t\cdot m$ and $(st)\cdot m = s\cdot (t\cdot m)$ for all $s,t\in S$ and $m\in M$;
	\item $s\cdot 0_M=0_M$ for all $s\in S$ and $0_S \cdot m = 0_M$ and $1\cdot m=m$ for all $m\in M$.
\end{itemize}

Let $S$ be a semiring and $A$ an $S$-semimodule. Also, let $(A, \cdot)$ be a commutative monoid. We recall that $A$ is said to be an $S$-semialgebra if for all $\alpha$ in $S$ and, $a$ and $b$ in $A$, we have \[\alpha (a \cdot b) = (\alpha a) \cdot b = a \cdot (\alpha b).\]

Following the terminology ``J\'{o}nsson module'' in \cite{GilmerHeinzer1983}, we say an $S$-semimodule $M$ is a J\'{o}nsson $S$-semimodule if the cardinality of $N$ is less than the cardinality of $M$ (i.e., $| N | < | M |$) for each proper $S$-subsemimodule $N$ of $M$ (Definition \ref{Jonssonsemimodule}), and then in Corollary \ref{Jonssonsemialgebra}, we show that each element of $A$ is either a unit or multiplicatively non-cancellative if $A$ is a J\'{o}nsson $S$-semialgebra.

After some preparations, we pass to the next section of the paper to investigate classical semirings. Let us recall that any finite ring is classical (Corollary \ref{finiteringclassical}). Surprisingly, it turns out that this is not the case in semiring theory (see Examples \ref{nonclassicalpropersemirings}). However, in Theorem \ref{afamilyfinitepropersemirings}, we prove that a famous family of finite semirings, i.e., $B(n,i)$s, are examples of proper classical semirings if $i > 0$. We also prove that a direct product of classical semirings is classical. This gives plenty of examples for classical semirings. 

Continuing our investigation of classical semirings in \S\ref{sec:classicalsemirings}, we also show that any complemented semiring is classical (see Theorem \ref{directproductclassicalsemirings} and Theorem \ref{complementedsemiringclassical}). We recall that an element $s$ of a semiring $S$ is complemented if there is an element $s^*$ such that $ss^* = 0$ and $s+s^* = 1$. A semiring $S$ is complemented if each element of $S$ is complemented \cite[\S5]{Golan1999}. Notice that Boolean algebras are important examples of complemented semirings.

Observe that in Lemma 3 in \cite{Gill1971}, Gill shows that a uniserial ring $R$ is classical if and only if $\Ann \Ann (b) = (b)$, for all $b \in R$. As a generalization to this result, we show that a uniserial semiring $S$ is classical if and only if $\Ann \Ann (b) = (b)$, for all $b \in S$ (see Theorem \ref{uniserialsemiringclassical}).

In Proposition \ref{expectationsemiringclassical}, we find a condition for the expectation semirings to be classical. In fact, we prove that if $S$ is a classical semiring and $M$ an $S$-semimodule such that $(M,+)$ is a group, then the expectation semiring $S \widetilde{\oplus} M$ is also classical.

From ring theory, we know that the total quotient ring $Q(R)$ of a ring $R$ is always classical and a ring $R$ is classical if and only if $R = Q(R)$ \cite[Proposition 11.4]{Lam1999}. It turns out that this is not the case in semiring theory. We summarize our findings in this direction as follows:

\begin{itemize}
	\item In Theorem \ref{totalquotientsemiringnotclassical}, we find a semiring $S$ such that $Q(S)$ is not classical.
	
	\item In Proposition \ref{totalquotientsemiringclassical}, we show that if $S$ is classical, then $S = Q(S)$.
	
	\item In Example \ref{totalquotientsemiringnotclassical2}, we find some semirings $S$ with $S = Q(S)$ while $S$ is not classical.
\end{itemize}

Since nilpotents are zero-divisors, semirings in which their elements are either units or nilpotents are examples of classical semirings. We call these semirings completely primary (see Definition \ref{completelyprimarysemiringdef}) and devote \S\ref{sec:completelyprimarysemirings} to discuss them. 

Note that the completely primary rings of the form \[\displaystyle \frac{\mathbb{F}_2[x_1,x_2,\dots,x_n]}{(x^2_1,x^2_2,\dots,x^2_n)},\] discussed in Proposition \ref{completelyprimaryrings}, have applications in coding theory \cite{DoughertyYildizKaradeniz2011}. However, our theory of completely primary semirings will be justifiable if we can find some examples of proper semirings of this kind. This is the task that we do in Proposition \ref{completelyprimarysemirings}. We also show that if $S$ is a completely primary semiring and $M$ an $S$-semimodule such that $(M,+)$ is a group, then the expectation semiring $S \widetilde{\oplus} M$ is also completely primary (see Proposition \ref{expectationsemiringcompletelyprimary}).

One may have seen in ring theory that any Artinian local ring is completely primary (see p. 136 in \cite{Cohn2003}). However, it turns out that there are some finite local semirings being not completely primary. Actually, in Proposition \ref{Artiniannotcompletelyprimary}, we show that if $S = \{0,u,1\}$ is a multiplicatively idempotent semiring with three elements such that $u + u \neq 1$ (for example, if $S$ is Hu's or LaGrassa's semiring explained in Examples \ref{nonclassicalpropersemirings}), then $S$ is local and Artinian but not completely primary.

\section{Semiring in which every non-unit is non-cancellative}\label{sec:nonunitsarenoncancellative}

Let us recall that an element $s$ of a semiring $S$ is multiplicatively cancellative if $sx_1 = sx_2$ implies $x_1 = x_2$, for all $x_1, x_2$ in $S$. The main purpose of this section is to study those semirings in which every non-unit is non-cancellative. Note that not all semirings have this property. For example, in the semiring of non-negative integers $\mathbb{N}_0$ equipped with ordinary addition and multiplication, each integer number $n > 1$ is neither a unit nor multiplicatively non-cancellative in $\mathbb{N}_0$.

Let us recall that a ring $R$ is called to be $\pi$-regular if for any $r \in R$, there is a $y \in R$ and a positive integer $n$ with $r^{n+1} y = r^n$ \cite[Theorem 3.1]{Huckaba1988}. As Kaplansky explains in \cite{Kaplansky1950}, $\pi$-regular rings were introduced by McCoy in \cite{McCoy1939}. Based on this, we give the following definition: 

\begin{definition}\label{piregulardef}
We say a semiring $S$ is $\pi$-regular if for any $s \in S$, there is a $t \in S$ and a positive integer $n$ such that $s^{n+1} t = s^n$.
\end{definition}

\begin{proposition}\label{piregularpro}
Let $S$ be a $\pi$-regular semiring. Then, each element of $S$ is either a unit or multiplicatively non-cancellative.
\end{proposition}

\begin{proof}
Consider an element $s$ in $S$. If $s$ is multiplicatively cancellative, then from $s^{n+1} t = s^n$, we obtain that $st = 1$ which means that $s$ is a unit. Otherwise, $s$ is multiplicatively non-cancellative and the proof is complete.
\end{proof}

\begin{examples}\label{periodicsemirings} In the following, we give some examples of $\pi$-regular semirings:
	
	\begin{enumerate}
		\item Any finite semiring $S$ is $\pi$-regular because $\{s^n\}_{n \in \mathbb{N}}$ has finitely many elements. Evidently, this implies that there are positive integers $m < n$ such that $s^m = s^n$.
		
		\item Let us recall that a semiring $S$ is (multiplicatively) periodic if for each $s$ in $S$ there are positive integers $m < n$ with $s^m = s^n$ \cite[p. 206]{Maletti2005}. It is clear that multiplicatively periodic semirings are $\pi$-regular.
		
		\item A ring $R$ is a simply periodic ring if for any $x\in R$, there is a positive integer $n > 1$ such that $x^n = x$ (see \cite{Mitchell1973} and \cite{Sussman1958}). Such rings have been discussed by Jacobson proving that any (associative) simply periodic ring is commutative (cf. Theorem 11 in \cite{Jacobson1945}). Similar to ring theory, we say a semiring $S$ is a simply periodic semiring if for any $s \in S$, there is a positive integer $n > 1$ such that $s^n = s$. It is evident that simply periodic semirings are $\pi$-regular. Note that multiplicatively idempotent semirings are examples of simply periodic semirings. 
	\end{enumerate}
	
\end{examples}

\begin{theorem}\label{eitherunitornoncancellative}
	Let $S$ be a semiring and $A$ an $S$-semialgebra. Suppose that the $S$-semimodule $A$ has no proper $S$-subsemimodule that is isomorphic to $A$. Then, each element of $A$ is either a unit or multiplicatively non-cancellative.
\end{theorem}

\begin{proof}
	Assume that $a$ is multiplicatively cancellative. Define $f: A \rightarrow A$ by $f(x) = ax$. First, we prove that $f$ is a semimodule homomorphism (i.e., an $S$-linear map). It is evident that $f$ preserves addition of $A$. Now, let $\alpha$ be an arbitrary element of $S$ and observe that by the properties of the $S$-semialgebra $A$, we have: \[f(\alpha x) = a(\alpha x) = \alpha (ax) = \alpha f(x).\] Now, since $a$ is multiplicatively cancellative, $f$ is injective. This means that $f(A)$ and $A$ are isomorphic $S$-semimodules. Since $f(A) \subseteq A$ and $A$ has no proper $S$-subsemimodule that is isomorphic to $A$, we have $f(A) = A$, i.e., $f$ is surjective. In particular, $f(b) = 1$ for some $b \in A$. This means that $ab = 1$ for some $b\in A$. Thus $a$ is a unit. This completes the proof. 	
\end{proof}

Following the terminology ``J\'{o}nsson algebra'' in \cite{ErdosHajnal1966} and ``J\'{o}nsson module'' in \cite{GilmerHeinzer1983}, we give the following definition:

\begin{definition}\label{Jonssonsemimodule} 
We say an $S$-semimodule $M$ is a J\'{o}nsson $S$-semimodule if the cardinality of $N$ is less than the cardinality of $M$ (i.e., $| N | < | M |$) for each proper $S$-subsemimodule $N$ of $M$. 
\end{definition}

\begin{corollary}\label{Jonssonsemialgebra}
	Let $S$ be a semiring and $A$ an $S$-semialgebra. Suppose that the $S$-semimodule $A$ is a J\'{o}nsson $S$-semimodule. Then, each element of $A$ is either a unit or multiplicatively non-cancellative.
\end{corollary}

\begin{proof}
	Because $A$ is a J\'{o}nsson $S$-semimodule, $A$ has no proper $S$-subsemimodule that is isomorphic to $A$. So by Theorem \ref{eitherunitornoncancellative}, each element of $A$ is either a unit or multiplicatively non-cancellative and this completes the proof.
\end{proof}

By definition, a ring $R$ is classical if each element of $R$ is either a unit or a zero-divisor \cite[p. 320]{Lam1999}. Since an element of a ring $R$ is non-cancellative if and only if it is a zero-divisor, we see that each element of a ring $R$ is either a unit or non-cancellative if and only if $R$ is classical. In view of Proposition \ref{piregularpro} and Theorem \ref{eitherunitornoncancellative}, we have the following:

\begin{corollary}\label{finiteringclassical}
A ring $R$ is classical if one of the following statements hold:

\begin{enumerate}
	\item $R$ is finite.
	\item $R$ is simply periodic (cf. \cite{Sussman1958}).
	\item $R$ is periodic (cf. \cite{Chacron1968}).
	\item $R$ is $\pi$-regular (cf. Theorem 3.1 in \cite{Huckaba1988}).
\end{enumerate}

Also, let $B$ be an $R$-algebra. Then, $B$ is classical if $B$ is a J\'{o}nsson $R$-module. In particular, if $K$ is a field and $A$ a finite dimensional $K$-algebra, then $A$ is classical \cite[Theorem 1.2.1]{DrozdKirichenko1994}. 
	
\end{corollary}

\begin{proof}
For the proof of the last statement, we recall that if $K$ is a field and $A$ is a finite dimensional $K$-vector space and $W$ is a subspace of $V$ with the same dimension, then $W=V$.
\end{proof}

\begin{remark}\label{applicationsclassicalrings} In the literature, classical rings are also called ``rings of quotients'' \cite[p. 320]{Lam1999} or ``full quotient rings'' \cite[\S3]{Ching1977}. Our motivation to study classical rings and their generalizations in semiring theory is their use in mathematics and applications in science and engineering as we explain in the following:
	\begin{enumerate}
		\item Classical rings are discussed in many resources in ring theory \cite[Theorem 3]{Hurley2006}, \cite[Lemma 3]{Gill1971}, \cite[p. 120]{Glaz1989}, \cite[Lemma 1]{NarkiewiczRuengsinsubLaohakosol2004}, and \cite{Stafford1982}. Note that classical rings have applications in other fields of mathematics like Manis valuation theory \cite[Lemma 3]{Zanardo1993}, number theory \cite[Proposition 4.1.8]{BuchmannVollmer2007}, and topological rings of Colombeau's generalized numbers (cf. \cite[Theorem 2.18]{AragonaJuriaans2001} and \cite[Lemma 3.2]{CortesFerreroJuriaans2009}).
		
		\item Classical rings have applications in different areas of science and engineering. For example, classical rings are discussed in multidimensional systems theory \cite[Definition 3.1]{Bose2003}, algebraic shift register sequences \cite[\S A.2.1]{GoreskyKlapper2012}, coding theory \cite{Blake1975,HurleyHurley2009}, system theory \cite{ChingWyman1977}, efficient information-theoretic multi-party computation \cite{EscuderoSoria2021}, and color visual cryptography schemes \cite{DuttaSardarAdhikariRujSakurai2020}.
	
	   \item Finally, classical rings have been discussed in mathematics education resources. For example, by studying rings in which every non-unit is a zero-divisor, students can gain a deeper understanding \cite{Cook2014} of these important mathematical structures, as well as their applications.
	
\end{enumerate}
\end{remark}

Inspired by the definition and properties of classical rings, we proceed to define classical semirings and investigate their properties in the next section:

\section{Classical semirings}\label{sec:classicalsemirings}

We collect the unit elements of a semiring $S$ in $U(S)$ and the zero-divisors in $Z(S)$.

\begin{definition}\label{classicalsemiringsdef}
We define a semiring $S$ to be classical if each element of $S$ is either a unit or a zero-divisor. In other words, a semiring $S$ is classical if $S = U(S) \cup Z(S)$.
\end{definition}

\begin{remark}
It is straightforward to see that if $S$ is a semiring, then $U(S)$ and $Z(S)$ are disjoint. However, this property does not necessarily hold for non-associative semirings, even if they are power-associative. For example, all nonzero elements of sedenions are invertible while infinitely many of them are zero-divisors as well \cite{Cawagas2004}. 
\end{remark}

The following result shows that not all semirings are classical:

\begin{proposition}\label{latticesnotclassical}
	Let $(C,\leq)$ be a chain with at least 3 elements such that $0$ is the smallest and $1$ is the greatest element of $C$. Then, $(C, \max, \min)$ is not a classical semiring though each element of $C$ is either a unit or multiplicatively non-cancellative.
\end{proposition}

\begin{proof}
Let $a \in C \setminus \{0,1\}$. Firstly, $a$ is not a zero-divisor because if $ab = 0$, then $b = 0$. On the other hand, if $ab = 1$, then $a = b = 1$. So, $a$ cannot be a unit. Thus $C$ is not classical. However, note that $C$ is multiplicatively idempotent, and so, $C$ is simply periodic. By Proposition \ref{piregularpro}, this implies that each element of $C$ is either a unit or multiplicatively non-cancellative. This completes the proof.
\end{proof} 

By Corollary \ref{finiteringclassical}, any finite ring is classical. Surprisingly, it turns out that this is not the case in semiring theory.

\begin{examples}\label{nonclassicalpropersemirings}
	In the following, we give some examples for finite proper semirings that are not classical:
	
	\begin{enumerate}
		
		\item Let us recall the definition of Hu's semiring given in \cite{Hu1975}. Let $H = \{0, u, 1\}$ be totally ordered by $0 < u < 1$ and define $\oplus$ on $S$ as \[a \oplus b = \max\{a,b\},\] except the case $a = b = 1$, where $1 \oplus 1 = u$. Then, it is easy to see that $(H, \oplus, \min)$ is a semiring \cite[Example 1.17]{Golan1999}. Since $u^2 = u$, $u$ is neither a unit nor a zero-divisor. This means that Hu's semiring is not classical.
		
		\item Let us recall the definition of LaGrassa's semiring given in \cite{LaGrassa1995}. Define the binary operations $\oplus$ and $\odot$ on $L = \{0,u,1\}$ in a way that they are both idempotent and we have the following: \[1 \oplus u = u \oplus 1 = u.\] It is easy to see that $(L, \oplus, \odot)$ is a semiring \cite[Example 6.4]{Golan1999}. Since $u^2 = u$, the element $u$ is neither a unit nor a zero-divisor. Therefore, LaGrassa's semiring is not classical.
		
		\item Let $X_n = \{-\infty, 0,1,\dots,n\}$, where $n$ is a positive integer and assume that $-\infty$ is the smallest element of $X_n$ such that \[-\infty + s = -\infty \qquad \forall~s\in X_n.\] Define addition and multiplication on $X_n$ as follows: \[x+y = \max\{x,y\} \text{~and~} xy = \min\{x+y,n\}.\] Then, $X_n$ equipped with above operations is a semiring (see \cite[Example 1.8]{Golan1999} and \cite[Example 1.9]{Smith1966}) such that its zero and one are $-\infty$ and 0, respectively. It is, now, easy to verify that $Z(X_n) = \{-\infty\}$ and $U(X_n) = \{0\}$. Since $n$ is a positive integer, $X_n$ has at least three elements and $X_n$ contains at least one element that is neither a unit and nor a zero-divisor. Consequently, for any positive integer $n$, the semiring $X_n$ is finite and proper but not classical. 
	\end{enumerate}
\end{examples}

There are some finite and proper semirings that are classical also. For example, the Boolean semiring $\mathbb{B} = \{0,1\}$ is a finite and proper semiring that is also classical. Now, we proceed to find more examples for classical proper semirings of finite order.

Assume that $n$ and $i$ are integer numbers such that $n \geq 2$, $0 \leq i \leq n-1$, and set \[B(n,i) = \{0,1,\dots, n-1\}.\] Define binary operations $\oplus$ and $\otimes$ on $B(n,i)$ as follows: \begin{equation*}
	x \oplus y =
	\begin{cases}
		x+y, &         \text{if } x+y \leq n-1,\\
		x+y \pmod {n-i}, &         \text{otherwise} 
	\end{cases}
\end{equation*}
and 
\begin{equation*}
	x \otimes y =
	\begin{cases}
		xy, &         \text{if } xy \leq n-1,\\
		xy \pmod {n-i}, &         \text{otherwise} 
	\end{cases}.
\end{equation*}

It is, then, easy to see that $(B(n,i), \oplus, \otimes)$ is a semiring \cite[Example 1.8]{Golan1999}. Note that the semiring $B(n,0)$ is the ring $\mathbb{Z}_n$ and $B(n,i)$ is a proper semiring if $1 \leq i \leq n-1$. As a generalization to this fact that $\mathbb{Z}_n$ is a classical ring, we show that $B(n,i)$ is also a classical semiring:

\begin{theorem}\label{afamilyfinitepropersemirings}
	Let $n\geq 2$ be a positive integer and $0 \leq i \leq n-1$. Then, $B(n,i)$ is a classical semiring. In particular, if $i > 0$, then $B(n,i)$ is an example of the family of finite proper classical semirings.
\end{theorem}

\begin{proof}
	Note that $B(n,0) = \mathbb{Z}_n$ is a classical ring \cite[Examples 11.6]{Lam1999}. Now, let $i > 0$ and $2 \leq x \leq n-1$ be an element of $B(n,i)$. If $x(n-i) > n-1$, then \[x(n-i) \equiv 0 \pmod {n-i}.\] If not, then we can find a suitable positive integer $m > 1$ such that $x^m(n-i) > n-1$, and so, \[x (x^{m-1}(n-i)) \equiv 0 \pmod {n-i}.\] This shows that each element of $B(n,i)$ except 1 is a zero-divisor. This implies that the only unit of $B(n,i)$ is the element 1. Thus $B(n,i)$ is a classical semiring, and it is proper if $i>0$. This completes the proof.
\end{proof}

The following result gives plenty of examples for classical semirings:

\begin{theorem}\label{directproductclassicalsemirings}
	Let $\Lambda$ be an index set and $S_i$ be a classical semiring for each $i \in \Lambda$. Then, $\prod_{i \in \Lambda} S_i$ is a classical semiring.
\end{theorem}

\begin{proof}
	It is clear that \[U\left(\prod_{i \in \Lambda} S_i\right) = \prod_{i \in \Lambda} U(S_i).\] We will be done if we prove that if at least one of the components of $(s_i)$ is not a unit, then $(s_i)$ is a zero-divisor. Assume that $s_{i_0}$ is not a unit in $S_{i_0}$ for some $i_0 \in \Lambda$. Therefore, by assumption, $s_{i_0}$ is a zero-divisor. This implies that there is a nonzero element $z_{i_0}$ in $S_{i_0}$ such that $s_{i_0} z_{i_0} = 0$. Now, define $(z_i)$ in $\prod_{i \in \Lambda} S_i$ such that $z_i = 0$ if $i \neq i_0$. It is clear that $(z_i)$ is nonzero in $\prod_{i \in \Lambda} S_i$ and $(s_i) (z_i) = (0)$. This shows that if an element in $\prod_{i \in \Lambda} S_i$ is not a unit, then it is a zero-divisor, and so, $\prod_{i \in \Lambda} S_i$ is classical and the proof is complete.
\end{proof}

Let us recall that an element $s$ of a semiring $S$ is complemented if there is an element $s^*$ such that $ss^* = 0$ and $s+s^* = 1$. A semiring $S$ is complemented if each element of $S$ is complemented \cite[\S5]{Golan1999}.

\begin{theorem}\label{complementedsemiringclassical}
Any complemented semiring is classical. 
\end{theorem}

\begin{proof}
Let $s \neq 1$ be an element of a complemented semiring $S$. Then, there is an element $s^*$ such that $ss^* = 0$ and $s+s^* = 1$. Since $s \neq 1$, $s^*$ needs to be nonzero. It follows that if $s \neq 1$, then $s$ is a zero-divisor. Since $1$ is a unit, we see that $S$ is classical and the proof is complete.
\end{proof}

Sarah Glaz finds some classical rings that are also reduced \cite[p. 120]{Glaz1989}. Let us recall that a semiring is nilpotent-free if 0 is its only nilpotent element. Now, we proceed to find some nilpotent-free classical proper semirings.

\begin{corollary}\label{booleanalgebraclassical}
Any Boolean algebra is a nilpotent-free classical semiring. 
\end{corollary}

\begin{proof}
By Theorem \ref{complementedsemiringclassical}, any Boolean algebra $B$ is classical. On the other hand, $b^2 = b$, for each $b \in B$. So, $B$ is nilpotent-free and this completes the proof.
\end{proof}

\begin{proposition}\label{directproductnilpotentfreeclassical}
	The direct product of classical nilpotent-free semirings is a classical nilpotent-free semiring.  
\end{proposition}

\begin{proof}
	It is clear that by Theorem \ref{directproductclassicalsemirings}, the direct product of classical semirings is classical. On the other hand, let $(s_i)$ be an element of the direct product of nilpotent-free semirings and assume that $(s_i)^n = (0)$. Therefore, $s^n_i = 0$ for each $i$, and so, $s_i = 0$, for each $i$. This means that the direct product of nilpotent-free semirings is nilpotent-free. This completes the proof.
\end{proof}

\begin{example} We give some examples of nilpotent-free proper semirings that are also classical. Let $\mathbb{B} = \{0,1\}$ be the Boolean semiring and $\{F_i\}_{i\in \Lambda}$ a family of fields. Then, \[S = \mathbb{B} \times \prod_{i\in \Lambda} F_i\] is a proper classical nilpotent-free semiring. Note that if $\Lambda$ is finite and each $F_i$ is a finite field, then $\mathbb{B} \times \prod_{i\in \Lambda} F_i$ is also finite.
\end{example}

Let us recall that the algebra of dual numbers is a 2-dimensional real algebra whose elements are of the form $a+b \epsilon$, where $a$ and $b$ are real numbers and $\epsilon^2 = 0$. The algebra of dual numbers is a powerful mathematical tool for kinematic and dynamic analysis of spatial mechanisms \cite{PennestriStefanelli2007}. We define generalized dual number algebras over an arbitrary ring as follows:

\begin{definition}\label{generalizeddualnumbers}
	Let $R$ be a ring and $M= \{0,1\} \cup \{s_i\}_{i=1}^n$ be a monoid with zero such that $0$ is its absorbing element and \[s_i s_j = 0, \qquad~\forall~i,j.\] Write elements of the generalized dual number algebra $R[S]$ over $R$ as follows: \[a_0 + \sum_{i=1}^n a_i s_i.\] In $R[S]$, let addition be component-wise and define multiplication as follows: \[\left(a_0 + \sum_{i=1}^n a_i s_i\right) \left(b_0 + \sum_{i=1}^n b_i s_i\right) = a_0 b_0 + \sum_{i=1}^n (a_0 b_i + a_i b_0)s_i.\]
\end{definition} 

\begin{theorem}\label{generalizeddualnumberalgebraclassical} 
	Let $R$ be a classical ring. The generalized dual number algebra $R[S]$ over $R$ defined in Definition \ref{generalizeddualnumbers} is a classical ring.
\end{theorem}

\begin{proof}
	Consider the generalized dual number $a_0 + \sum_{i=1}^n a_i s_i$, where $a_i$s are arbitrary elements of the classical ring $R$. An easy computation shows that if $a_0$ is invertible, then the multiplicative inverse of $a_0 + \sum_{i=1}^n a_i s_i$ is computed as follows: \[\frac{1}{a_0 + \sum_{i=1}^n a_i s_i} = \frac{1}{a_0} + \sum_{i=1}^n \frac{-a_i}{a^2_0}s_i.\] On the other hand, if $a_0$ is a zero-divisor, then there is a nonzero element $x$ in $R$ such that $a_0x = 0$. Now, observe that \[ \left(a_0 + \sum_{i=1}^n a_i s_i\right) \cdot \left(\sum_{i=1}^n x s_i\right) = 0. 
	\] Thus the generalized dual numbers $R[S]$ over $R$ is a classical ring and the proof is complete.
\end{proof}

A nonempty subset $I$ in a semiring $S$ is an ideal of $S$ if $a+b \in I$ and $sa \in I$, for all $a, b \in I$ and $s \in S$. An ideal $I$ is proper if $I \neq S$. A proper ideal $P$ of $S$ is prime if $ab \in P$ implies that either $a \in P$ or $b \in P$, for all $a,b \in S$. All prime ideals of a semiring $S$ is collected in $\Spec(S)$. Maximal ideals of a semiring $S$ is collected in $\Max(S)$. Note that if $S$ is a semiring, then $\Max(S)$ is a nonempty subset of $\Spec(S)$ (see Proposition 6.59 and Corollary 7.13 in \cite{Golan1999}). Similar to ring theory, if one collects all nilpotent element of a semiring $S$ in a set $\Nil(S)$, then $\Nil(S)$ is the intersection of all prime ideals of the semiring $S$ \cite[Proposition 7.28]{Golan1999}.

\begin{lemma}\label{metasemifieldpro}
	Let $S$ be a semiring. Then, the following statements are equivalent:
	
	\begin{enumerate}
		\item Each element of $S$ is either a unit or a nilpotent;
		
		\item $\Nil(S)$ is the only prime ideal of $S$;
		
		\item $\Nil(S)$ is a maximal ideal of $S$. 
	\end{enumerate}
\end{lemma}

\begin{proof}
	$(1) \implies (2)$: Let $S$ be a semiring such that each element of $S$ is either a unit or a nilpotent. We know that $\Nil(S) = \bigcap_{\mathfrak{p} \in \Spec(S)} \mathfrak{p}$ and $U(S) = S \setminus \bigcup_{\mathfrak{m} \in \Max(S)} \mathfrak{m}$ \cite[Proposition 6.61]{Golan1999}. Since each prime ideal is included in a maximal ideal \cite[Proposition 6.59]{Golan1999}, we have \[\bigcap_{\mathfrak{p} \in \Spec(S)} \mathfrak{p} = \bigcup_{\mathfrak{p} \in \Spec(S)} \mathfrak{p}.\] Therefore, $S$ has only one prime ideal, and so, $\Nil(S)$ is the only prime ideal of $S$.
	
	$(2) \implies (3)$: Obvious. 
	
	$(3) \implies (1)$: If $\Nil(S)$ is a maximal ideal, then it is the only maximal ideal of $S$. On the other hand, \[U(S) = S \setminus \bigcup_{\mathfrak{m} \in \Max(S)} \mathfrak{m} = S \setminus \Nil(S).\] This implies that each element of $S$ is either a unit or a nilpotent and the proof is complete.
\end{proof}

The concept of localization in semiring theory is defined similar to its counterpart in ring theory \cite{Kim1985}. Similar to ring theory, one can prove that if $\mathfrak{p}$ is a prime ideal of a semiring $S$, then the localization $S_\mathfrak{p}$ of $S$ at $S\setminus \mathfrak{p}$ is a local semiring and its unique maximal ideal is $\mathfrak{p} S_\mathfrak{p}$. Also, it is easy to observe that there is a one-to-one correspondence between the prime ideals of the local semiring $S_\mathfrak{p}$ and those prime ideals of $S$ contained in $\mathfrak{p}$. 	

\begin{theorem}\label{semiringskrulldimensionzeroclassical}
	Let $S$ be a semiring such that each prime ideal of $S$ is maximal, i.e., the Krull dimension of $S$ is zero. Then, $S$ is classical. 
\end{theorem}

\begin{proof}
	Let $S$ be a semiring such that each prime ideal of $S$ is maximal and $\mathfrak{m}$ a maximal ideal of $S$. Since each prime ideal of $S$ is maximal, the only prime ideal of the local semiring $(S_{\mathfrak{m}}, \mathfrak{m} S_{\mathfrak{m}})$ is $\mathfrak{m} S_{\mathfrak{m}}$. By Lemma \ref{metasemifieldpro}, any $s \in \mathfrak{m}$ is nilpotent in $S_{\mathfrak{m}}$, and so, there is a positive integer $n$ such that $s^n = 0$ in $S_{\mathfrak{m}}$. This implies that for each $s \in \mathfrak{m}$ there is a $t \in S \setminus \mathfrak{m}$ such that $t s^n = 0$. We may choose the positive integer $n$ such that $ts^{n-1} \neq 0$. This shows that if an element $s$ of the semiring $S$ is not a unit, then it is a zero-divisor. Hence, $S$ is classical, as required.
\end{proof}

Let $B$ be a nonempty subset of a semiring $S$. The annihilator of $B$, denoted by $\Ann(B)$, is defined to be the set of all elements $s$ in $S$ such that $sb = 0$, for all $b \in B$. If $B = \{b\}$, then $\Ann(B)$ is simply denoted by $\Ann(b)$. An ideal $I$ of $S$ is principal if it is generated by a single element $s$ of $S$. In this case, we denote $I$ by $(s)$.   

Valuation theory for rings was introduced by Krull in \cite{Krull1939}. Some parts of the valuation theory for semirings were introduced and discussed in \cite{Nasehpour2018}. Let us recall that a semidomain is a valuation semiring if its ideals are linearly ordered by inclusion \cite[Theorem 2.4]{Nasehpour2018}. More generally, an arbitrary semiring is uniserial if its ideals are linearly ordered by inclusion \cite[Definition 2.2]{BehzadipourNasehpour2023}. Therefore, any (valuation) uniserial semiring is local, i.e., it has only one maximal ideal. The following is a generalization of Lemma 3 in \cite{Gill1971}:

\begin{theorem}\label{uniserialsemiringclassical}
	A uniserial semiring $S$ is classical if and only if $\Ann \Ann (b) = (b)$, for all $b \in S$.
\end{theorem}

\begin{proof}
	Let $S$ be a classical semiring. It is clear that $(b) \subseteq \Ann \Ann (b)$, for all $b \in S$. Now, assume that there is a $c \in S$ and a nonzero $y \in S$ such that $y \in \Ann \Ann (c) \setminus (c)$. Since $S$ is a uniserial semiring, $(c) \subset (y)$. This implies that $c = my$, where $m$ is an element of the only maximal ideal $\mathfrak{m}$ of the semiring $S$. Observe that \[(y) \subseteq \Ann \Ann (c) = \Ann \Ann (my).\] Since annihilator operator is an inclusion reversing and a tripotent operator, by taking annihilator from both sides of the above inclusion, we have \[\Ann (y) \supseteq \Ann (my).\] This means that $smy = 0$ implies $sy = 0$, for all $s \in S$. Therefore, \[(y) \cap \Ann(m) = \{0\}.\] On the other hand, $(y)$ is a nonzero ideal. Now, since ideals of $S$ are totally ordered by inclusion, $\Ann(m) = \{0\}$. This means that $m$ is not a zero-divisor. Since $S$ is classical, $m$ needs to be a unit which is a contradiction because $m$ is an element of the maximal ideal $\mathfrak{m}$. Thus $\Ann \Ann (b) = (b)$, for all $b \in S$.
	
	Conversely, let $\Ann \Ann (b) = (b)$, for all $b \in S$ and $m$ be an element of $S \setminus Z(S)$. This implies that $\Ann (m) = \{0\}$. Therefore, $\Ann \Ann (m) = S$. So, $(m) = S$ which means that $m$ is a unit and the proof is complete.  
\end{proof}

Let $S$ be a semiring and $M$ an $S$-semimodule. On the set $S \times M$, define addition and multiplication as follows:

\begin{itemize}
	\item $(s_1, m_1) + (s_2,m_2) = (s_1 + s_2 , m_1 + m_2)$,
	
	\item $(s_1, m_1) \cdot (s_2, m_2) = (s_1 s_2, s_1 m_2 + s_2 m_1)$.
\end{itemize}

The set $S \times M$ equipped with above operations is a semiring denoted by $S \widetilde{\oplus} M$ and called the expectation semiring of the $S$-semimodule $M$ \cite[Example 7.3]{Golan2003}. The additively invertible elements of an $S$-semimodule $M$ is collected in the set $V(M)$.

\begin{proposition}\label{expectationsemiringclassical}
Let $S$ be a semiring and $M$ an $S$-semimodule such that $(M,+)$ is a group, i.e., $V(M) = M$. If $S$ is classical then the expectation semiring $S \widetilde{\oplus} M$ is also classical.
\end{proposition}

\begin{proof}
Let $M$ be an $S$-semimodule such that $V(M) = M$. Assume that $(s,m)$ is not a unit element of the expectation semiring $S \widetilde{\oplus} M$. Then, by Theorem 2.2 in \cite{Nasehpour2020}, $s$ is not a unit. Since $S$ is classical, $s$ has to be a zero-divisor of $S$. Again, by Theorem 2.2 in \cite{Nasehpour2020}, $(s,m)$ is a zero-divisor of $S \widetilde{\oplus} M$. Hence, $S \widetilde{\oplus} M$ is classical. This completes the proof.
\end{proof}

Let $R$ be a ring. The total quotient ring $Q(R)$ of $R$ is a classical ring and a ring $R$ is classical if and only if $R = Q(R)$ \cite[Proposition 11.4]{Lam1999}. This is perhaps why classical rings are also known as ``rings of quotients'' \cite[p. 320]{Lam1999}) or ``full quotient rings'' \cite[\S3]{Ching1977} in the literature.

Now, let $S$ be a semiring and $MC(S)$ the set of multiplicatively cancellative elements of $S$. Then, the localization of $S$ at $MC(S)$ \cite[\S11]{Golan1999}, denoted by $Q(S)$ and called the total quotient semiring, is the counterpart of total quotient ring. The question arises if the semiring $Q(S)$ is classical for an arbitrary semiring $S$. Before answering to this question, we give the following result:

\begin{lemma}\label{totalquotiententire} 
	Let $S$ be an entire semiring. Then, $Q(S)$ is entire.
\end{lemma}

\begin{proof}
	Let $a/b$ and $c/d$ be elements in $Q(S)$ with $(a/b)(c/d) = 0/0$. This implies that $ac = 0$, since $c$ and $d$ are multiplicatively cancellative. Since $S$ is entire, we have either $a = 0$ or $c = 0$. This means that either $a/b = 0/1$ or $c/d = 0/1$. Thus $Q(S)$ is entire and the proof is complete.
\end{proof}

\begin{theorem}\label{totalquotientsemiringnotclassical}
	Let $(R,\oplus,\odot)$ be a commutative ring with $1\neq 0$ such that there is a non-trivial zero-divisor $n \in R$. Assume that $z \notin R$, and set $S = R \cup \{z\}$. Define addition $+$ and multiplication $\cdot$ on $S$ as follows:
	
	\begin{enumerate}
		\item If $s,t \in R$, then $s + t = s \oplus t$, and if $s \in S$, then $s + z = z + s = s$.
		
		\item If $s,t \in R$, then $s \cdot t = s \odot t$, and if $s \in S$, then $s \cdot z = z \cdot s = z$.
	\end{enumerate}
	
	Then, $Q(S)$ is not a classical semiring.  
\end{theorem}

\begin{proof}
	Observe that $(S,+,\cdot, z,1)$ is an entire semiring (see Example 1.6 in \cite{Golan1999} and Example 1.1 in \cite{LaGrassa1995}). It is clear that $r \in R$ is multiplicatively cancellative if and only if $r$ is regular. This implies that $MC(S) = R \setminus Z(R)$. Since $S$ is entire, by Lemma \ref{totalquotiententire}, $Q(S)$ is entire. Therefore, $Z(Q(S)) = \{z/1\}$. Let $n$ be a non-trivial zero-divisor of the ring $R$ and consider the element $n/1$ in $Q(S)$. Our claim is that $n/1$ is neither a unit nor a zero-divisor. It is clear that $n/1$ is nonzero in $Q(S)$, and so, is not a zero-divisor in $Q(S)$ because $Q(S)$ is entire. Also, note that $n/1$ is not a unit in $Q(S)$ because if there is an element $x/y \in Q(S)$ such that \[n/1 \cdot x/y = 1/1,\] then $nx = y$. Since $y$ is not a zero-divisor in $R$, $n$ is not a zero-divisor in $R$, a contradiction. Hence, $Q(S)$ is not classical, as required.
\end{proof}

\begin{proposition}\label{totalquotientsemiringclassical}
Let $S$ be a semiring. Then, the following statements hold:
	
\begin{enumerate}
	\item The semiring $S$ can be considered as a subsemiring of $Q(S)$.
	\item If $S$ is classical, then $S =Q(S)$.
\end{enumerate}
	
\end{proposition}

\begin{proof}
	(1): It is easy to verify that the function $S \rightarrow Q(S)$ defined by $s \mapsto s/1$ is injective. Therefore, $S$ can be considered as a subsemiring of $Q(S)$.
	
	(2): Now, let $S$ be classical. Consider an arbitrary element $x/s \in Q(S)$, where $x \in S$ and $s \in MC(S)$. Since $s \in MC(S)$, $s$ cannot be a zero-divisor. Since $S$ is classical, $s$ is a unit. Now, since $x/s = s^{-1} x$ and $ s^{-1} x \in S$, we see that $Q(S) \subseteq S$. Now, by (1), we see that $S = Q(S)$ and the proof is complete. 
\end{proof}

Let $R$ be a ring. It is easy to see that if $R = Q(R)$, then $R$ is a classical ring. Now, let $S$ be a semiring such that $S = Q(S)$. The question arises if $S$ is classical. The following examples show that the answer to this question is negative.

\begin{example}\label{totalquotientsemiringnotclassical2}
	Let $S$ be a proper non-classical semiring such that it is multiplicatively idempotent, i.e., for each $u \in S$, we have $u^2 = u$ (consider Hu's semiring or LaGrassa's semiring given in Examples \ref{nonclassicalpropersemirings}). It is clear that $MC(S) = \{1\}$, and so, $S = Q(S)$ while $S$ is not classical.
\end{example}

A semiring $S$ is called to be a principal ideal semiring (for short, PIS) if each ideal of $S$ is principal \cite{Nasehpour2019}.

It is a famous fact in semiring theory that if $S$ is a commutative semiring, then the set of all ideals $\Id(S)$ of the semiring $S$ is a commutative semiring where addition and multiplication of ideals of $S$ are defined as following: \[I+J = \{a+b: a\in I \text{~and~} b\in J\}\] \[IJ = \left\{\sum_{i=1}^{n} a_i b_i: a_i \in I,~b_i \in J,~ n \in \mathbb{N}\right\}.\] Note that the only unit element of $\Id(S)$ is $S$. First we prove the following:

\begin{proposition}\label{semiringidealsclassical}
Let $S$ be a semiring. Then, the following statements hold:
\begin{enumerate}
	\item If $\Id(S)$ is classical, then so is the semiring $S$.
	\item If $S$ is classical and a PIS, then $\Id(S)$ is also classical. 
\end{enumerate}

\end{proposition}

\begin{proof}
(1): Let $\Id(S)$ be classical. If $s \in S$ is not a unit in $S$, then the principal ideal $(s)$ is a proper ideal of $S$, and so, it cannot be a unit in $\Id(S)$ because the only unit element of the semiring $\Id(S)$ is $S$. Now, since $\Id(S)$ is classical, the ideal $(s)$ is a zero-divisor on $\Id(S)$ which means that there is a nonzero ideal $J$ of $S$ such that $(s)J = (0)$. Let $b$ be a nonzero element of $J$. Therefore, $sb = 0$, and so, $s$ is a zero-divisor on $S$. Thus $S$ is classical.

(2): Now, let $S$ be classical. Our purpose is to prove that $\Id(S)$ is classical. Since by assumption, $S$ is a PIS, each element of $\Id(S)$ is of the form $(s)$. If $(s)$ is not a unit, then $(s) \neq S$ and this is equivalent to say that $s$ is not a unit in $S$. So, there is a nonzero element $t \in S$ such that $st = 0$. This implies that $(s)(t) = (0)$. Therefore, $(s)$ is a zero-divisor on $\Id(S)$. Thus $\Id(S)$ is classical and the proof is complete. 
\end{proof}

\begin{corollary}
If $R$ is a classical and principal ideal ring, then $\Id(R)$ is a classical semiring. Moreover, if $R$ is a finite principal ideal ring, then $\Id(R)$ is a finite classical semiring.
\end{corollary}

\begin{example}
Let $n > 1$ be an integer. The ring $\mathbb{Z}_n$ is a finite principal ideal ring. So, $\Id(\mathbb{Z}_n)$ is an example of a finite classical semiring.
\end{example}

Since each nilpotent element of a semiring is a zero-divisor, semirings in which their elements are either units or nilpotents are examples of classical semirings. We pass to the next section to discuss this subfamily of classical semirings.

\section{Completely primary semirings}\label{sec:completelyprimarysemirings}

Let us recall that in commutative ring theory, a ring $R$ is called to be a completely primary ring if $\Nil(R)$ is a maximal ideal \cite[Definition 1.2]{Snapper1950}. A ring $R$ is completely primary if and only if each element of $R$ is either a unit or a nilpotent (see Exercise 7.4.40 in \cite{DummitFoote2004}). Completely primary rings have been discussed in some papers related to commutative ring theory (cf. \cite{ClarkGosaviPollack2017} and \cite{Rahimi2003}). Inspired by this and in view of Lemma \ref{metasemifieldpro}, we give the following definition:

\begin{definition}\label{completelyprimarysemiringdef}
We say a semiring $S$ is completely primary if each element of $S$ is either a unit or a nilpotent.
\end{definition}

\begin{remark}
Rings whose elements are either units or nilpotents have application in automata theory \cite{ArvindMukhopadhyaySrinivasan2010}. They have been investigated in non-commutative algebra \cite{Kashan2021} also. Note that if $R$ is a ring and an $R$-module $M$ is Artinian, Noetherian, and indecomposable, then each element the ring $\End_R(M)$ (which is usually non-commutative) is either a unit or a nilpotent (check Lemma 3.25 on p. 77 in \cite{Magurn2002}). In non-commutative algebra, such rings are sometimes called metadivision (check Definition 1 in \cite{LuHe2014}).
\end{remark}

\begin{proposition}\label{completelyprimaryrings}
Let $K$ be a field, $\{x_i\}_{i\in \Lambda}$ indeterminates over $K$, and $m_i$ be a positive integer for each $i \in \Lambda$. Consider the ideal $I = (x^{m_i}_i : i \in \Lambda)$ of the polynomial ring $R = K[x_i : i \in \Lambda]$. Then, $R/I$ is a completely primary ring.
\end{proposition}

\begin{proof}
Let $\mathfrak{p}$ be a prime ideal of $R$ containing $I$. Since $x^{m_i}_i$ is an element of $\mathfrak{p}$ and $\mathfrak{p}$ is prime, $\mathfrak{p}$ contains the maximal ideal $\mathfrak{m} = (x_i : i \in \Lambda)$ of $R$ and so $\mathfrak{p} = \mathfrak{m}$. This means that the only prime ideal of $R/I$ is $\mathfrak{p}/I$. Hence, by Lemma \ref{metasemifieldpro}, $R/I$ is completely primary, as required.
\end{proof}

\begin{remark}
The completely primary ring $\displaystyle \frac{\mathbb{F}_2[x_1,x_2,\dots,x_n]}{(x^2_1,x^2_2,\dots,x^2_n)}$ has applications in algebraic coding theory \cite{DoughertyYildizKaradeniz2011}.
\end{remark}

Now, we proceed to find proper completely primary semirings.

\begin{proposition}\label{completelyprimarysemirings}
Let $(P,+,0)$ be an idempotent commutative monoid and set $S=P \cup \{1\}$. Let us extend addition on $S$ as $a+1=1+a=1$ for all $a\in S$ and define multiplication over $S$ as $ab=0$ for all $a,b \in P$ and $a\cdot 1=1 \cdot a=a$ for all $a\in S$. Then, $S$ is a completely primary proper semiring.  
\end{proposition}

\begin{proof}
By Proposition 20 in \cite{Nasehpour2016}, $S$ is a semiring such that the only unit element of $S$ is 1. Also, since $ab=0$, for all $a,b \in P$, $a^2 = 0$, for each $a\in P$. Therefore, all elements of $S$ except 1 are nilpotent. Since $S$ is an additively idempotent semiring, $S$ is not a ring. This means that $S$ is a proper semiring such that each element of $S$ is either a unit or a nilpotent. This completes the proof.
\end{proof}

\begin{proposition}\label{expectationsemiringcompletelyprimary}
Let $S$ be a semiring and $M$ an $S$-semimodule such that $(M,+)$ is a group, i.e., $V(M) = M$. Then, the semiring $S$ is completely primary if and only if the expectation semiring $S \widetilde{\oplus} M$ is completely primary.
\end{proposition}

\begin{proof}
Let $M$ be an $S$-semimodule such that $V(M) = M$. By Theorem 2.2 in \cite{Nasehpour2020}, we have the following: \[U(S \widetilde{\oplus} M) = U(S) \widetilde{\oplus} M \text{~and~} \Nil(S\widetilde{\oplus} M) = \Nil(S) \widetilde{\oplus} M.\] Hence, the semiring $S$ is completely primary if and only if the expectation semiring $S \widetilde{\oplus} M$ is completely primary, as required.	
\end{proof}

In the following, we describe how to construct proper classical semirings such that they are not completely primary.

\begin{theorem}\label{directproductcompletelyprimarysemirings}
	Let $\Lambda$ be an index set with at least two elements and $S_i$ be a completely primary semiring for each $i \in \Lambda$. Then, $\prod_{i \in \Lambda} S_i$ is classical but not completely primary.
\end{theorem}

\begin{proof}
	By Theorem \ref{directproductclassicalsemirings}, $\prod_{i \in \Lambda} S_i$ is a classical semiring. However, if $(s_i)$ is an element of $\prod_{i} S_i$ such that each $s_i$ is 1 except one of them which is zero, then $(s_i)$ is neither a unit nor a nilpotent. This shows that $\prod_{i} S_i$ is not completely primary and the proof is complete.
\end{proof}

One may have seen in ring theory that any Artinian local ring is completely primary (see p. 136 in \cite{Cohn2003}). So, it is quite natural to see if an Artinian local semiring $S$ is completely primary. It turns out that this is not the case even if $S$ is finite:  

\begin{proposition}\label{Artiniannotcompletelyprimary}
Let $S = \{0,u,1\}$ be a multiplicatively idempotent semiring with three elements such that $u + u \neq 1$ (for example, let $S$ be Hu's or LaGrassa's semiring explained in Examples \ref{nonclassicalpropersemirings}). Then, $S$ is local and Artinian but not completely primary. 
\end{proposition}

\begin{proof}
Since $S$ is finite, it is obviously Artinian. Note that the only non-trivial ideal of $S$ is the principal ideal generated by $u$ which is definitely the only maximal ideal of $S$. So, $S$ is local. Clearly, $S$ is not completely primary because $u^2 = u$ which means that $u$ is neither a unit nor a nilpotent. This finishes the proof. 
\end{proof}

\begin{proposition}\label{semiringidealscompletelyprimary}
	Let $S$ be a semiring. Then, the following statements hold:
	\begin{enumerate}
		\item If $\Id(S)$ is completely primary, then so is the semiring $S$.
		\item If $S$ is completely primary and a PIS, then $\Id(S)$ is also completely primary. 
	\end{enumerate}
	
\end{proposition}

\begin{proof}
	(1): Let $\Id(S)$ be completely primary. If $s \in S$ is not a unit in $S$, then the principal ideal $(s)$ is a proper ideal of $S$, and so, it cannot be a unit in $\Id(S)$ because the only unit element of the semiring $\Id(S)$ is $S$. Now, since $\Id(S)$ is completely primary, the ideal $(s)$ is nilpotent on $\Id(S)$ which means that there is positive integer such that $(s)^n = (0)$. This evidently implies that $s^n = 0$. Therefore, $s$ is nilpotent. Thus $S$ is completely primary.
	
	(2): Now, let $S$ be completely primary. Our purpose is to prove that $\Id(S)$ is completely primary. Since by assumption, $S$ is a PIS, each element of $\Id(S)$ is of the form $(s)$. If $(s)$ is not a unit, then $(s) \neq S$ and this is equivalent to say that $s$ is not a unit in $S$. So, $s^n = 0$, for some positive integer $n$. This implies that $(s)^n = (0)$. Therefore, $(s)$ is nilpotent. Thus $\Id(S)$ is completely primary and the proof is complete. 
\end{proof}

\begin{corollary}
	If $R$ is a completely primary and principal ideal ring, then $\Id(R)$ is a completely primary semiring.
\end{corollary}

\begin{example}
Let $p$ be a prime number and $k$ a positive integer. Since $\mathbb{Z}_{p^k}$ is a completely primary and principal ideal ring, $\Id(\mathbb{Z}_{p^k})$ is a completely primary semiring.
\end{example}

\begin{remark}\label{monoidswithzero}
Let us recall that a monoid $(M,\cdot, 1)$ is a monoid with zero if there is an element 0 in $M$ such that \[m \cdot 0 = 0 \cdot m = 0, \qquad\forall~m \in M.\] Some of the results of the current paper such as Proposition \ref{piregularpro}, Proposition \ref{latticesnotclassical}, Theorem \ref{directproductclassicalsemirings}, Proposition \ref{directproductnilpotentfreeclassical}, and Theorem \ref{directproductcompletelyprimarysemirings} can be easily stated and proved for commutative monoids with zero. 
\end{remark}


\bibliographystyle{plain}

\begin{thebibliography}{15.}
	
\bibitem{AragonaJuriaans2001} Aragona, J., Juriaans, S.O.: Some structural properties of the topological ring of Colombeau's generalized numbers. Commun. Algebra 29, No. 5, 2201--2230 (2001).

\bibitem{ArvindMukhopadhyaySrinivasan2010} Arvind, V., Mukhopadhyay, P., Srinivasan, S.: New results on noncommutative and commutative polynomial identity testing. Comput. Complexity 19, No. 4, 521--558 (2010).

\bibitem{BehzadipourNasehpour2023} Behzadipour, H., Nasehpour, P.: Some remarks on the comparability of ideals in semirings. J. Algebra Appl. Vol. 16, No. 1, Article ID 2350006, 14 p. (2023). 
	
\bibitem{Blake1975} Blake, I.F.: Codes over integer residue rings. Inf. Control 29, 295--300 (1975).

\bibitem{Bose2003} Bose, N.K.: Multidimensional systems theory and applications. Dordrecht: Kluwer Academic Publishers. xix, 269 p. (2003).

\bibitem{BuchmannVollmer2007} Buchmann, J., Vollmer, U.: Binary quadratic forms. An algorithmic approach. Algorithms and Computation in Mathematics 20. Berlin: Springer. xiv, 318 p. (2007).

\bibitem{Cawagas2004} Cawagas, R.E.: On the structure and zero divisors of the Cayley-Dickson sedenion algebra. Discuss. Math., Gen. Algebra Appl. 24, No. 2, 251--265 (2004).

\bibitem{Chacron1968} Chacron, M.: Certains anneaux p\'{e}riodiques. Bull. Soc. Math. Belg. 20, 66--78 (1968).

\bibitem{Ching1977} Ching, W.-S.: Linear equations over commutative rings. Linear Algebra Appl. 18, 257--266 (1977).

\bibitem{ChingWyman1977} Ching, W.-S., Wyman, B.F.: Duality and the regulator problem for linear systems over commutative rings. J. Comput. Syst. Sci. 14, 360--368 (1977).

\bibitem{ClarkGosaviPollack2017} Clark, P.L., Gosavi, S., Pollack, P.: The number of atoms in a primefree atomic domain. Commun. Algebra 45, No. 12, 5431--5442 (2017).

\bibitem{Cohn2003} Cohn, P.M.: Further algebra and applications. London: Springer. xii, 451 p. (2003).


\bibitem{Cook2014} Cook, J.P.: The emergence of algebraic structure: students come to understand units and zero-divisors, Int. J. Math. Educ. Sci. Technol. 45, No. 3, 349--359 (2014).

\bibitem{CortesFerreroJuriaans2009} Cortes, W., Ferrero, M., Juriaans, S.O.: The Colombeau quaternion algebra. Giambruno, Antonio (ed.) et al., Groups, rings and group rings. International conference, Ubatuba, Brazil, July 28 -- August 2, 2008. Providence, RI: American Mathematical Society (AMS). Contemporary Mathematics 499, 37--45 (2009).

\bibitem{DoughertyYildizKaradeniz2011} Dougherty, S.T., Yildiz, B., Karadeniz, S.: Codes over $R_k$, Gray maps and their binary images. Finite Fields Appl. 17, No. 3, 205--219 (2011).

\bibitem{DrozdKirichenko1994} Drozd, Y.A., Kirichenko, V.V.: Finite dimensional algebras. With an appendix by Vlastimil Dlab. Transl. from the Russian by Vlastimil Dlab. Berlin: Springer-Verlag. xiii, 249 p. (1994).

\bibitem{DummitFoote2004} Dummit, D.S., Foote, R.M.: Abstract algebra. 3rd ed. Wiley International Edition. Chichester: Wiley. xii, 932 p. (2004).

\bibitem{DuttaSardarAdhikariRujSakurai2020} Dutta, S., Sardar, M.K., Adhikari, A., Ruj, S., Sakurai, K.: Color Visual Cryptography Schemes Using Linear Algebraic Techniques over Rings. In: Kanhere, S., Patil, V.T., Sural, S., Gaur, M.S. (eds) Information Systems Security. ICISS 2020. Lecture Notes in Computer Science, vol 12553. Springer, Cham (2020).

\bibitem{ErdosHajnal1966} Erd\"{o}s, P., Hajnal, A.: On a problem of B. Jonsson. Bull. Acad. Pol. Sci., S\'{e}r. Sci. Math. Astron. Phys. 14, 19--23 (1966).

\bibitem{EscuderoSoria2021} Escudero, D., Soria-Vazquez, E.: Efficient information-theoretic multi-party computation over non-commutative rings. Malkin, Tal (ed.) et al., Advances in cryptology – CRYPTO 2021. 41st annual international cryptology conference, CRYPTO 2021, virtual event, August 16--20, 2021. Proceedings. Part II. Cham: Springer. Lect. Notes Comput. Sci. 12826, 335--364 (2021).

\bibitem{Gill1971} Gill, D.T.: Almost maximal valuation rings. J. Lond. Math. Soc., II. Ser. 4, 140--146 (1971).

\bibitem{GilmerHeinzer1983} Gilmer, R., Heinzer, W.: On J\'{o}nsson modules over a commutative ring. Acta Sci. Math. 46, 3--15 (1983).

\bibitem{Glaz1989} Glaz, S.: Commutative coherent rings. Lecture Notes in Mathematics. 1371. Berlin etc.: Springer-Verlag. xi, 347 p. (1989).

\bibitem{Golan2003} Golan, J.S.: Semirings and affine equations over them: theory and applications. Mathematics and its Applications (Dordrecht) 556. Dordrecht: Kluwer Academic Publishers. xiv, 241 p. (2003).

\bibitem{Golan1999} Golan, J.S.: Semirings and their applications. Dordrecht: Kluwer Academic Publishers. xi, 381 p. (1999).

\bibitem{GoreskyKlapper2012} Goresky, M., Klapper, A.: Algebraic shift register sequences. Cambridge: Cambridge University Press. xv, 498 p. (2012).

\bibitem{Hu1975} Hu, S.C.: A ternary algebra for probability computation of digital circuits, in Proceedings of the 1975 International Symposium on Multiple-Valued Logic, IEEE Computer Society, Long Beach (1975).

\bibitem{Huckaba1988} Huckaba, J.A.: Commutative rings with zero divisors. Monographs and Textbooks in Pure and Applied Mathematics, 117. New York etc.: Marcel Dekker, Inc. X, 216 p. (1988).

\bibitem{HurleyHurley2009} Hurley, P., Hurley, T.: Codes from zero-divisors and units in group rings. Int. J. Inf. Coding Theory 1, No. 1, 57--87 (2009).

\bibitem{Hurley2006} Hurley, T.: Group rings and rings of matrices. Int. J. Pure Appl. Math. 31, No. 3, 319--335 (2006).

\bibitem{Jacobson1945} Jacobson, N.: Structure theory for algebraic algebras of bounded degree, Annals of Mathematics, Second Series, Vol. 46, No. 4, 695--707 (1945).

\bibitem{Kaplansky1950} Kaplansky, I.: Topological representation of algebras. II. Trans. Am. Math. Soc. 68, 62-75 (1950).

\bibitem{Kashan2021} Kashan, H.A.: Rings whose elements are a sum of a unit regular and nilpotent. Ital. J. Pure Appl. Math. 46, 645--658 (2021).

\bibitem{Kim1985} Kim, C.B.: A note on the localization in semirings. J. Sci. Inst. Kookmin Univ. 3, 13--19 (1985).

\bibitem{Krull1939} Krull, W.: Beitr\"{a}ge zur Arithmetik kommutativer Integrit\"{a}tsbereiche. VI: Der allgemeine Diskriminantensatz. Unverzweigte Ringerweiterungen. Math. Z. 45, 1--19 (1939).

\bibitem{LaGrassa1995} LaGrassa, S.: Semirings: Ideals and Polynomials, PhD . Thesis, University of Iowa (1995).

\bibitem{Lam1999} Lam, T.Y.: Lectures on modules and rings. Graduate Texts in Mathematics. 189. New York, NY: Springer. xxiii, 557 p. (1999).

\bibitem{LuHe2014} Lu, J., He, L.: On the structures of Abelian $\pi$-regular rings. Int. J. Math. Math. Sci. 2014, Article ID 842313, 4 p. (2014).

\bibitem{Magurn2002} Magurn, B.A.: An algebraic introduction to $K$-theory. Encyclopedia of Mathematics and Its Applications. 87. Cambridge: Cambridge University Press. xiv, 676 p. (2002).

\bibitem{Maletti2005} Maletti, A.: Hasse diagrams for classes of deterministic bottom-up tree-to-tree-series transformations. Theor. Comput. Sci. 339, No. 2--3, 200--240 (2005).

\bibitem{McCoy1939} McCoy, N.H.: Generalized regular rings. Bull. Am. Math. Soc. 45, 175--178 (1939).

\bibitem{Mitchell1973} Mitchell, W.: Simple periodic rings. Pac. J. Math. 44, 651--658 (1973).

\bibitem{NarkiewiczRuengsinsubLaohakosol2004} Narkiewicz, W., Ruengsinsub, P, Laohakosol, V.: An addendum to the paper: ``Arithmetic functions over rings with zero divisors'' by Ruangsinsap, Laohakosol and P. Udomkavanich. Bull. Malays. Math. Sci. Soc. (2) 27, No. 1, 87--90 (2004).

\bibitem{Nasehpour2020} Nasehpour, P.: Algebraic properties of expectation semirings. Afr. Mat. 31, No. 5--6, 903--915 (2020).

\bibitem{Nasehpour2016} Nasehpour, P.: On the content of polynomials over semirings and its applications. J. Algebra Appl. 15, No. 5, Article ID 1650088, 32 p. (2016).

\bibitem{Nasehpour2019} Nasehpour, P.: Some remarks on semirings and their ideals. Asian-Eur. J. Math. 12, No. 7, Article ID 2050002, 14 p. (2019).

\bibitem{Nasehpour2018} Nasehpour, P.: Valuation semirings. J. Algebra Appl. 17, No. 4, Article ID 1850073, 23 p. (2018).

\bibitem{PennestriStefanelli2007} Pennestr\`{i}, E.; Stefanelli, R.: Linear algebra and numerical algorithms using dual numbers. Multibody Syst. Dyn. 18, No. 3, 323--344 (2007).

\bibitem{Rahimi2003} Rahimi, A.M.: Some results on $n$-stable rings. Missouri J. Math. Sci. 15, No. 2, 11 p. (2003).

\bibitem{Smith1966} Smith, F.A.: A structure theory for a class of lattice ordered semirings. Fundam. Math. 59, 49--64 (1966).

\bibitem{SokratovaKaljulaid2000} Sokratova, O., Kaljulaid, U.: $\Omega$-rings and their flat representations. Dorninger, D. (ed.) et al., Contributions to general algebra 12. Proceedings of the 58th workshop on general algebra ``58. Arbeitstagung Allgemeine Algebra'', Vienna, Austria, June 3--6, 1999. Klagenfurt: Verlag Johannes Heyn. 377--390 (2000).

\bibitem{Snapper1950} Snapper, E.: Completely primary rings. I. Ann. Math. (2) 52, 666--693 (1950).

\bibitem{Stafford1982} Stafford, J.T.: Noetherian full quotient rings. Proc. Lond. Math. Soc., III. Ser. 44, 385--404 (1982).

\bibitem{Sussman1958} Sussman, I.: A generalization of Boolean rings. Math. Ann. 136, 326--338 (1958).

\bibitem{Zanardo1993} Zanardo, P.: Constructions of Manis valuation rings. Commun. Algebra 21, No. 11, 4183--4194 (1993).

\end{thebibliography}

\end{document}